\documentclass[11pt,letterpaper]{amsart}

\usepackage{amssymb,amsthm,amsmath}
\usepackage[numbers,sort&compress]{natbib}
\usepackage{color}
\usepackage{graphicx}
\usepackage{tikz}
\usepackage{amssymb,amsthm,amsmath}
\usepackage[numbers,sort&compress]{natbib}
\usepackage{color}
\usepackage{graphicx}
\usepackage{tikz}

\title[ ]{ Proof of  geometric    Borg's Theorem in arbitrary dimensions }

\author{Wencai Liu}
\address[W. Liu]{ Department of Mathematics, Texas A\&M University, College Station, TX 77843-3368, USA} \email{liuwencai1226@gmail.com; wencail@tamu.edu}



\theoremstyle{plain}
\newtheorem{theorem}{Theorem}[section]

\newtheorem{lemma}[theorem]{Lemma}

\newtheorem{remark}{Remark}
\newcommand{\C}{\mathbb{C}}

\newcommand{\Z}{\mathbb{Z}}

\newcommand{\R}{\mathbb{R}}

\theoremstyle{plain}

\newtheorem{conjecture}{Conjecture}

\begin{document}
	
	
	\begin{abstract}
		Let  $\Delta+V$    be the discrete Schr\"odinger operator,  where $\Delta$ is the discrete Laplacian on $\mathbb{Z}^d$ and  potential $V:\mathbb{Z}^d\to \C$ is $\Gamma$-periodic with  $\Gamma=q_1\mathbb{Z}\oplus q_2 \mathbb{Z}\oplus\cdots\oplus q_d\mathbb{Z}$.  In this study, we  establish a  comprehensive characterization of  complex-valued $\Gamma$-periodic functions such that the Bloch variety of $\Delta+V$ contains a graph of an entire function, in  particular,  we show that  there are exactly $q_1q_2\cdots q_d$  such functions (up to  Floquet isospectrality and  translation). Moreover, 	by applying this understanding  to 	real-valued functions $V$,    we prove that $V$ is constant if and only if the Bloch variety of $\Delta+V$ contains a graph of an entire function, which   confirms  the conjecture concerning  the geometric version of Borg's theorem in arbitrary dimensions.

	\end{abstract}
	
	\maketitle 
	\section{Introduction and main results}

	Given $q_l\in \Z_+$, $l=1,2,\cdots,d$,
	let $\Gamma=q_1\Z\oplus q_2 \Z\oplus\cdots\oplus q_d\Z$.
	A function $V: \Z^d\to \C$ is said to be $\Gamma$-periodic if for any $\gamma\in \Gamma$ and $n\in\Z^d$, $V(n+\gamma)=V(n)$.

	Let $\Delta$ be the discrete Laplacian on $\ell^2(\Z^d)$, namely
	\begin{equation*}
	(\Delta u)(n)=\sum_{||n^\prime-n||_1=1}u(n^\prime),
	\end{equation*}
	where $n=(n_1,n_2,\cdots,n_d)\in\Z^d$, $n^\prime=(n_1^\prime,n_2^\prime,\cdots,n_d^\prime)\in\Z^d$ and 
	\begin{equation*}
	||n^\prime-n||_1=\sum_{l=1}^d |n_l-n^\prime_l|.
	\end{equation*}
	Consider the discrete  Schr\"{o}dinger operator on $\ell^2({\Z}^d)$,
	\begin{equation} \label{h0}
	H=\Delta +V ,
	\end{equation}
	where $V$ is $\Gamma$-periodic.

Denote by $\{\textbf{e}_j\}$, $j=1,2,\cdots d$ the standard basis in $\Z^d$.

  Floquet theory leads to the study of equation 
		\begin{equation}   \label{spect_0}
		(\Delta u)(n)+V(n)u(n)=\lambda u(n), n\in\Z^d,
	\end{equation}
with the so called Floquet-Bloch boundary condition
	\begin{equation}  
	u(n+q_j\textbf{e}_j)=e^{2\pi i k_j}u(n),j=1,2,\cdots,d, \text{ and } n\in \Z^d.\label{Fl}
	\end{equation}

 Following this, we introduce a fundamental domain $W$ for $\Gamma$:
\begin{equation*}
W=\{n=(n_1,n_2,\cdots,n_d)\in\Z^d: 0\leq n_j\leq q_{j}-1, j=1,2,\cdots, d\}.
\end{equation*}
By writing out $H=\Delta +V$ as acting on the $Q=q_1q_2\cdots q_d$ dimensional space 
$\{u=\{ u(n)\}_{n\in W}: u(n)\in\C\}$, the equation \eqref{spect_0} with boundary conditions \eqref{Fl} 
translates into the eigenvalue problem for a $Q\times Q$ matrix ${D}_V(k)$, where $k=(k_1,k_2,\cdots,k_d)$.

Assume that $V$ is real. 
For each $k\in\R^d$,  $D_V(k)$ has $Q=q_1q_2\cdots q_d$ eigenvalues. Order them in non-decreasing  order
\begin{equation*}
\lambda_V^1(k)\leq \lambda_V^2(k)\leq\cdots \leq \lambda_V^Q(k).
\end{equation*}
We call   $\lambda_V^m(\cdot)$ the $m$-th (spectral) band function, $m=1,2,\cdots, Q$.

Denote by 
	\begin{align*}
[a_V^m,b^m_V]= [\min_{k\in\R^d}\lambda_V^m (k),\max_{k\in\R^d}\lambda_V^m (k)], m=1,2,\cdots, Q.
\end{align*}

For a real function $V$, 
the spectrum $\Delta+V$  is the union of the spectral band $[a_V^m,b_V^m]$, $m=1,2,\cdots, Q$:
\begin{equation}\label{gband}
\sigma (\Delta+V)={\bigcup }_{m= 1}^{Q}[a_V^m,b_V^m].
\end{equation}

If $b_V^m<a_V^{m+1}$ for some $m=1,2,\cdots,Q-1$, $(b_V^m, a_V^{m+1}) $  is called a spectral gap.
 
 This paper primarily focuses on discrete periodic Schr\"odinger operators, yet we will also  discuss the history and developments of continuous periodic Schr\"odinger operators.

We start with a review of the classical Borg's Theorem in both  continuous \cite{borg} and discrete  (e.g. \cite[Theorem 5.4.21]{s2011} and \cite[Theorem 3.6]{GKTBook}) cases.

[{\bf Classical Borg's Theorem}] 
Let $d=1$.  Assume that $V$ is a   real-valued periodic function.  Then the following statements are equivalent:
\begin{enumerate}
	\item The potential $V$ is  a constant function.
	\item  $  \Delta+V$ has no spectral gaps.
\end{enumerate}

In modern proofs of Borg’s uniqueness theorem for the one-dimensional (matrix-valued) case, the periodicity condition on the potential is replaced by the reflectionless property. For details on these methods, see ~\cite{ges1,ges2,ges3}, where the trace formula serves as a key component in the proofs.  
Moreover, Borg’s uniqueness theorem can be viewed as a special case of the broader finite-gap spectral theory—a rich and elegant branch of inverse spectral theory~\cite{Sim11,SY97}.


Despite the extensive development of Borg-type results in one dimension, the situation in higher dimensions is fundamentally different. In particular, the analogue of Borg’s theorem fails for $d \geq 2$.
Indeed,  for $d\geq2$, 
there are many non-constant (small) real periodic functions $V$ such that   $\Delta+V$ has no spectral gaps (e.g. continuous case~\cite[Theorem 6.1]{ksurvey} and discrete case~\cite{hj18,ef,fk2}).  
In higher dimensions, it is quite common to see spectral bands overlapping. This is addressed in the Bethe-Sommerfeld Conjecture (e.g. \cite{ps,kar04,par08,vel}).   So, spectral gaps provide much less information on potentials in higher dimensions.

   Denote by    $B(H)\subset \C^d\times\C$  Bloch variety of $H=\Delta+V$: 
 \begin{equation}
 B(H)=\{(k,\lambda) \in\C^d\times\C: \det(D_V(k)-\lambda I)=0 \}.
 \end{equation} 
 
In \cite{as,ktcmh90,ksurvey},  Borg's theorem was reformulated geometrically in terms of the Bloch variety $B(H)$ of $H=\Delta+V$, which could be generalized to arbitrary dimensions.
  This gives rise to  the following conjecture. 
  
  {\bf Conjecture 1}  \cite[Conjecture 5.39]{ksurvey}.
  
  Assume that $V$ is a real-valued periodic function.  Then the following statements are equivalent:
  \begin{enumerate}
  	\item The potential $V$ is  a constant function.
  	\item There exists an entire function $f(k)$ such that  for all $k\in\C^d$, $(k,f(k))\in B(H)$.
  \end{enumerate}
Conjecture 1 is commonly known as the geometric version of Borg's theorem. This conjecture was also mentioned in \cite{as,ktcmh90}. 
 It has been discussed in  \cite{as,ktcmh90}  that the statement that for  real-valued periodic functions $V$ in one dimension,  the function $V$ is constant if and only if  there exists an entire function $f(k)$ such that  $(k,f(k))\in B(H)$  for all $k\in\C^d$ is equivalent to  the classical Borg's theorem.

The celebrated work of  Kn\"orrer-Trubowitz proves Conjecture 1 for $d=2$, as an application of directional compactification of the Bloch variety\cite{ktcmh90}.

In the present work, 
 we expand our investigation to encompass   complex-valued potentials. 
 We establish a criterion that enables us to determine whether the Bloch variety $B(\Delta+V)$, associated with a complex-valued function $V$ in arbitrary dimensions, contains  a graph $(k,f(k))$ of an entire function $\lambda=f(k)$. 
 As an application of this criterion to real potentials, we prove Conjecture 1.

  \begin{theorem}\label{mainthm2}
 	Assume that $V$ is a  complex-valued  $\Gamma$-periodic function.
 	Then the  following statements are equivalent:
 	\begin{enumerate}
 		\item  There exist $l\in W$ and a constant $K$ such that 
 		\begin{equation}\label{g43}
 		\det (D_V(k)-\lambda I ) = \prod_{n\in W}\left( K -\lambda+\sum_{j=1}^d \left(e^{2\pi \frac{n_j+k_j}{q_j} i} +e^{-2\pi \frac{n_j+l_j+k_j}{q_j} i}\right) \right).
 		\end{equation}
 		\item There exists an entire function $f(k)$ such that  $(k,f(k))\in B(H)$.
 	\end{enumerate}

 \end{theorem}
 \begin{remark}
 	\begin{enumerate}
 		\item Denote by ${\bf K}$ the constant function whose value is $K$ (regarding  ${\bf K}$ as a  $\Gamma$-periodic function). 
 		Direct computations (e.g. see  Lemma \ref{lesep} below) imply that
 		\begin{equation}\label{g451}
 		\det (D_{\bf K}(k)-\lambda I ) = \prod_{n\in W}\left( K  -\lambda+\sum_{j=1}^d \left(e^{2\pi \frac{n_j+k_j}{q_j} i} +e^{-2\pi \frac{n_j+k_j}{q_j} i}\right) \right).
 		\end{equation}
 		\item Setting  $l_j = 0 $  for all  $j=1,2,\cdots,d$ in equation \eqref{g43} yields the right side of \eqref{g451}.
 	\end{enumerate}
 \end{remark}
 Denote by $\sigma( D_{V} (k)) $ the (counting the algebraic multiplicity)  eigenvalues of $D_{V} (k)$. Two $\Gamma$-periodic potentials $V$ and $Y$ are called 
 Floquet isospectral if  
 \begin{equation}\label{gfi}
 \sigma(D_{V} (k))= \sigma(D_{Y} (k)), \text{ for all } k \in\R^d.
 \end{equation}
 
 For periodic Schr\"odinger operators, $V$ and $Y$ are  
 Floquet isospectral  if and only if the Bloch varieties $B(\Delta+V)$ and $B(\Delta+Y)$ are the same.
 
Combining Theorem \ref{mainthm2} with some basic facts of Floquet isospectrality, we have

 \begin{theorem}\label{mainthm1}
  Assume that $V$ is a  real-valued $\Gamma$-periodic function.
Then the following statements are equivalent:
\begin{enumerate}
	\item The potential $V$ is  a constant function.
	\item There exists an entire function $f(k)$ such that  $(k,f(k))\in B(H)$.
\end{enumerate}
 
 \end{theorem}
Denote by $X$ the space of all complex-valued $\Gamma$-periodic functions $V$. We define an equivalence relation  $ \sim$ on $ X$: we say $V\sim Y$ if and only if $V$ and $Y$ are Floquet isospectral.

Denote by $X_e$ all complex-valued $\Gamma$-periodic functions $V$ that have zero mean and the Bloch variety $B(\Delta+V)$ contains the graph of an entire function. More precisely, $V\in X_e$ if and only if $\sum_{n\in W} V(n)=0$ and there exists an entire function $f_V(k)$ such that for all $k\in\C^d$, $ (k,f_V(k))\in B(\Delta+V)$.
    \begin{theorem}\label{mainthm4}
We have  that
\begin{equation}\label{gnew1}
\# \{ X_e/\sim\}=Q,
\end{equation}
    and
    \begin{equation}\label{gnew58}
    \# \{ X_e\}\leq Q Q!.
    \end{equation}

   \end{theorem}
   \begin{remark}
The equation \eqref{gnew1} in Theorem \ref{mainthm4} says that  up to the Floquet isospectrality, there are exactly $Q$  complex-valued $\Gamma $-periodic functions of mean zero whose Bloch varieties contain  the graphs of entire functions. 
   \end{remark}

  Before diving into the challenges and solutions that this paper presents, it's crucial to underscore  the significance of Bloch varieties.
  Analytic and algebraic properties of Bloch   varieties and their associated Fermi varieties, $F_{\lambda} (H)={k\in\C^d: (k,\lambda)\in B(H)}$,  play a   pivotal role in the study of  periodic   operators. 
  They contribute significantly to understanding aspects such as embedded eigenvalues, isospectrality, and quantum ergodicity \cite{LiuPreprint:fermi, liufloquet23, GKTBook,kv06cmp,kvcpde20,liu1,shi1,liuqua22,ms22,sabri2023flat,F2021Ballistic,fk1,fk2,shijst,lishi,flshi}.  For an extensive understanding and background,  we refer readers to the surveys \cite{ksurvey,liujmp22,kuchment2023analytic}.

  Bloch and Fermi varieties provide a framework to reformulate the (inverse) spectral problems concerning periodic Schr\"odinger operators. This paper, for instance, geometrically recasts Borg's theorem in any dimension. 
 Another instance is the author's recent work,  where he  used  these varieties to  establish  several rigidity theorems which partially depend on the reinterpretation of (inverse) spectral problems~\cite{liufloquet23,LiuPreprint:fermi}.
 
 The benefits of such reformulations are multifold. Spectral theory concerning real potentials is notably more developed than complex potentials. By reformulating spectral and inverse spectral problems via Bloch and Fermi varieties, we can shift our focus towards the spectral theory of complex-valued potentials as opposed to merely real potentials. As an example, both Theorem 1.1 in this paper and the isospectrality results in \cite{LiuPreprint:fermi} allow potentials to be complex-valued.
 Moreover, these reformulations using Bloch and Fermi varieties provide opportunities to employ various tools from algebraic and analytic geometry,  and  multi-variable complex analysis,  to study spectral  problems arising from periodic operators.

 Switching to the topic of our  proof, consider $z_j=e^{2\pi i k_j},j=1,2,\cdots,d$ and $\mathcal{D}_V(z)= D_V(k)$. For discrete periodic operators, it is known that $\det (D_V(k)-\lambda I)$ simplifies to a Laurent polynomial $\mathcal P_V(z,\lambda)=\det(\mathcal{D}_V(z)-\lambda I)$ after changing variables.

  Dating back to  1980s,  B\"atting, Gieseker, Kn\"{o}rrer, and Trubowitz ~\cite{GKTBook,ktcmh90,battig1988toroidal,batcmh92,bktcm91} employed compactification to explore Bloch and Fermi varieties, thereby successfully proving Conjecture 1 for $d=2$ and  deriving irreducibility results in dimensions $d=2$ and $3$. The author introduced a novel approach  to prove the irreducibility of a family of Laurent polynomials, leading to the confirmation of two irreducibility conjectures of Bloch and Fermi varieties in arbitrary dimensions \cite{liu1}.  Some ideas in ~\cite{liu1} have been generalized  to periodic graph operators \cite{FLM1,FLM2}.

  Inspired by the proofs developed in \cite{liu1, LiuPreprint:fermi}, this paper   focuses on the study of the Laurent polynomial $\mathcal P_V(z,\lambda)$,  rather than  the compactification approach used in \cite{ktcmh90}.
  Our strategy in proving Theorem \ref{mainthm2} involves the application of multi-variable complex analysis and perturbation theory to establish the asymptotics of eigenvalues within an appropriate domain $\Omega$.  In  this domain, we  show that all eigenvalues of $\mathcal{D}_V(z)$ are distinct,  and hence  eigenvalues of $\mathcal{D}_V(z)$ are holomorphic in $\Omega$.
 To carry out the plan, we restrict $(z_2,z_3,\cdots,z_d)$  to  a suitably bounded domain and allow $z_1$ to approach infinity. Subsequently, we perform a Laurent series expansion of eigenvalues within the domain $\Omega$ with respect to $z_1$. The asymptotics of eigenvalues enable us to obtain that all Laurent coefficients of $z_1$ vanish, except for those of degrees $1,0,-1$. The process eventually leads to the proof of Theorem \ref{mainthm2}. When $V$ is real, by applying Theorem \ref{mainthm2}, we deduce that $V$ is Floquet isospectral to a constant potential. Finally, an Ambarzumian type theorem (aka  inverse spectral theorem for free Schr\"odinger operators) concludes Theorem \ref{mainthm1}.

  The proof of Theorem \ref{mainthm4} significantly relies on the combination of tools from spectral analysis and algebraic geometry. Utilizing spectral analysis, we are able to reduce the validation of Theorem \ref{mainthm4} to problems in algebraic geometry: study the solutions of a system consisting of $Q$ polynomial equations with $Q$ variables. It is important to note that this system exhibits the asymptotics of $Q$ elementary symmetric polynomials.
  To tackle this challenge, we apply well-established principles from algebraic geometry  to our algebraic equations. This enables us to show that our target system is always solvable and has a finite number (with an explicit bound) of solutions. As a result, we are able to conclude the proof of Theorem \ref{mainthm4}.

 Finally, we would like to compare the current paper with two previous works by the author \cite{liu1, LiuPreprint:fermi}. In \cite{liu1}, the author  studied  the irreducibility of  Bloch and Fermi varieties. Furthermore, in \cite{LiuPreprint:fermi}, the author introduced a new type of inverse spectral problem called Fermi isospectrality and established several rigidity theorems.
  In the present paper, our focus is on investigating whether the Bloch varieties contain graphs of entire functions. As an application of our main results, we have successfully proved the geometric Borg's theorem in  arbitrary dimensions. Although all three papers involve the study of the Laurent polynomial $\mathcal P(z,\lambda)$, the subjects explored in each paper are fundamentally different. Moreover, the specific approaches and proofs developed in this paper are entirely new.

 	The rest of this paper is structured as follows. 
 	Section \ref{S2} revisits some basics   related to  discrete periodic Schr\"odinger operators. 
 	In Section \ref{S3},  we provide several  technical lemmas about eigenvalue perturbations. 
 In Section \ref{S4}, we   complete  the proof of  Theorems \ref{mainthm2} and \ref{mainthm1}.
 Finally, Section \ref{S5} provides  the proof of  Theorem \ref{mainthm4}.

		\section{ Basics}\label{S2}
	
	In this section, we  revisit some basic facts about the  discrete periodic Schr\"odinger operators, see, e.g., \cite{liu1,ksurvey,liujmp22}.
Define the discrete Fourier transform $\hat{V}(l) $ for $l\in {W}$ (note that we identify  $W$ with its standard dual lattice) by 
\begin{equation*}
\hat{V}(l) =\frac{1}{{Q}}\sum_{ n\in {W} } V(n) \exp\left\{-2\pi i \left(\sum_{j=1}^d \frac{l_j n_j}{q_j} \right)\right\}.
\end{equation*}
 and  extend $\hat{V}(l)$ to $ \Z^d$    periodically, namely, for any $l\equiv m\mod \Gamma$,
\begin{equation*}
\hat{V}(l)=\hat{V}(m).
\end{equation*}

	Let $\C^{\star}=\C\backslash \{0\}$.  
Recall that  $z_j=e^{2\pi i k_j}$, $j=1,2,\cdots,d$,  $z=(z_1,z_2,\cdots,z_d)$,  $\mathcal{D}_V(z)=D_V(k)$ and $\mathcal{P}_V(z,\lambda)=\det (\mathcal{D}_V(z)-\lambda I)$.

Define
\begin{equation}\label{gtm}
\tilde{\mathcal{D}}_V(z)= \tilde{\mathcal{D}}_V(z_1,z_2,\cdots,z_d)= \mathcal{D}_V(z_1^{q_1},z_2^{q_2},\cdots,z_d^{q_d}),
\end{equation}
and 
\begin{equation}\label{gtp}
\tilde{\mathcal{P}}_V(z,\lambda)=\det( \tilde{\mathcal{D}}_V(z)-\lambda I)= \mathcal{P}_V(z_1^{q_1},z_2^{q_2},\cdots,z_d^{q_d},\lambda).
\end{equation}
Let $$\rho^j_{m}=e^{2\pi  \frac{m}{q_j} i},$$
where $0\leq m \leq q_j-1$, $j=1,2,\cdots,d$.

 We now state the following lemma:

\begin{lemma}\label{lesep} (e.g. \cite[Lemma 4.1]{liu1})
	Let $n=(n_1,n_2,\cdots,n_d) \in {W}$ and $n^\prime=(n_1^\prime,n_2^\prime,\cdots,n_d^\prime) \in {W}$. Then 
	$\tilde{\mathcal{D}}_V(z)$ is unitarily equivalent to 		
	$
	A(z)+B_V,
	$
	where $A(z)$ is a diagonal matrix with entries
	\begin{equation}\label{A}
	A(z,n,n^\prime)=\left(\sum_{j=1}^d \left(\rho^j_{n_j}z_j+\rho^j_{-n_j}z_j^{-1} \right)\right) \delta_{n,n^{\prime}}
	\end{equation}
	and \begin{equation}\label{gb}
	B_V(n,n^\prime)=\hat{V} \left(n_1-n_1^\prime,n_2-n_2^\prime,\cdots, n_d-n_d^\prime\right).
	\end{equation}
	In particular,
	\begin{equation*}
	\tilde{\mathcal{P}}_V(z, \lambda) =\det(A(z)+B_V-\lambda I).
	\end{equation*}

\end{lemma}
\section{Eigenvalue perturbations}\label{S3}

 \begin{lemma}\label{leper}
 	Let $M=(M_{ij})$ be a complex
 	$L\times L$ matrix.
 	Denote by $B(M_{ii}, R_i) =\{x\in\C: |x-M_{ii}|\leq R_i\}$, $i=1,2,\cdots,L$, the Gershgorin discs, where
 	$R_i=\sum_{j\neq i}|M_{ij}|$.
 	Assume that the discs $B(M_{ii}, R_i)$, $i=1,2,\cdots,L$ are disjoint.
 	Then $M$ has exactly one eigenvalue in each disc $B(M_{ii}, R_i)$, $i=1,2,\cdots,L$.
 	
 \end{lemma}

\begin{proof}
	This follows from the standard Gershgorin circle theorem.
\end{proof}

Let $C_1$ be a large constant depending on $||V||$ and $\Gamma$. For $d\geq 2$,
denote by $\hat{z}=(z_2,z_3,\cdots,z_d)$.
Let $\hat{\Omega}=\{\hat{z} \in \C^{d-1}: C_1^{d-j+1}\leq  |z_j|\leq C_1^{d-j+1}+1, j=2,3,\cdots,d\}$.
Let  $\Omega=\{z\in(\C^{\star}) ^d: |z_1| \geq C_1^{d}, \hat{z}\in \hat{\Omega}\}$ for $d\geq 2$,
and $\Omega=\{z\in\C^{\star} : |z| \geq C_1\}$ for $d=1$.

In the following $O(1)$ is bounded only depending on $||V||$ and $\Gamma$. We emphasize that  $C_1>O(1)$.

For  $d\geq2$, denote by $\hat W=\{n=(n_1,n_2,\cdots,n_d)\in  W:  n_1=0 \}$.
\begin{lemma}\label{ledif}
	For any distinct $l\in W$ and $l'\in W$,  we have that  
		\begin{equation}\label{g52}
	\left|\left(\sum_{j=1}^d \rho_{l_j}^j z_j\right)-\left(\sum_{j=1}^d \rho_{l'_j}^j z_j\right)\right| \geq\frac{1}{2} \min_{j\in\{1,2,\cdots,d\} } \left\{\left|1-e^{\frac{2\pi i}{q_j}}\right|\right\}C_1, z\in \Omega.
	\end{equation} 
	 
		For any distinct $l\in \hat W$ and $l'\in \hat W$,  we have that  
		\begin{equation}\label{g51}
		\left|\left(\sum_{j=2}^d \rho_{l_j}^j z_j\right)-\left(\sum_{j=2}^d \rho_{l'_j}^j z_j\right)\right| \geq \frac{1}{2} \min_{j\in\{2,3,\cdots,d\} } \left\{\left|1-e^{\frac{2\pi i}{q_j}}\right|\right\}C_1,\hat z \in \hat{\Omega}.
		\end{equation} 
		
 
\end{lemma}
\begin{proof}
It is easy to see that \eqref{g51} follows from \eqref{g52} by setting $l_1=l_1'=0$. So in order to prove Lemma \ref{ledif},  it suffices to prove \eqref{g52}.
Let $m$ be the smallest natural number in $\{1,2,\cdots,d\}$ such that $l_m\neq l_m'$. Then
\begin{align}
&	\left|\left(\sum_{j=1}^d \rho_{l_j}^j z_j\right)-\left(\sum_{j=1}^d \rho_{l'_j}^j z_j\right)\right| \nonumber\\&=	\left|\left(\sum_{j=m}^d \rho_{l_j}^j z_j\right)-\left(\sum_{j=m}^d \rho_{l'_j}^j z_j\right)\right| \nonumber\\
	&\geq  |\rho_{l_m}^m- \rho_{l_m'}^m | | z_m|  -	\left|\left(\sum_{j=m+1}^d \rho_{l_j}^j z_j\right)-\left(\sum_{j=m+1}^d \rho_{l'_j}^j z_j\right)\right| .\label{g25j1}
\end{align}
Since for $z \in \Omega$, we have $|z_1| \geq C_1^d$ and
$C_1^{d-j+1} \leq |z_j| \leq C_1^{d-j+1} + 1$ for  $j = 2, 3, \ldots, d,$
it follows from \eqref{g25j1} that
\begin{align*}
	\left|\left(\sum_{j=1}^d \rho_{l_j}^j z_j\right)-\left(\sum_{j=1}^d \rho_{l'_j}^j z_j\right)\right| &
\geq  |\rho_{l_m}^m- \rho_{l_m'}^m| |z_m|  -	4d C_1^{-1}|z_m|,\\
&=  |1- \rho_{l_m'-l_m}^m| |z_m|  -	4d C_1^{-1}|z_m|.
\end{align*}
This implies \eqref{g52}. 
\end{proof}

\begin{lemma}\label{lemper1}
	Let $d\geq 2$. Let $\hat{A}(\hat z)$ be a diagonal matrix, and $\hat{B}_V$ be a Toeplitz matrix, given by the following: for any $n\in \hat W$ and $n'\in \hat W$,
	\begin{equation}\label{A1}
	\hat{A}(\hat z, n,n^\prime)=\left(\sum_{j=2}^d \left(\rho^j_{n_j}z_j+\rho^j_{-n_j}z_j^{-1} \right)\right) \delta_{n,n^{\prime}},
	\end{equation}
	and \begin{equation}\label{gb1}
	\hat{B}_V(n,n^\prime)=\hat{V} \left(0,n_2-n_2^\prime,\cdots, n_d-n_d^\prime\right).
	\end{equation}
	Then, for $\hat z\in\hat \Omega$, the matrix $\hat{A}(\hat z) +\hat{B}_V$ has exactly $q_2q_3\cdots q_d$ distinct eigenvalues, 
	$\hat{\lambda}_V^l(\hat z)$,  $l\in \hat W$, with the following properties: $\hat{\lambda}_V^l(\hat z)$, $l\in \hat W$, is holomorphic in $\hat{\Omega}$ and
	\begin{equation}\label{g28}
	\hat{\lambda}_{V}^l(\hat z)=\left( \sum_{j=2}^d \rho^j_{l_j}z_j\right)+O(1).
	\end{equation}
\end{lemma}
\begin{proof}
	
	According to Lemmas \ref{lesep} and \ref{leper},  and \eqref{g51}, for any $\hat z\in \hat{\Omega}$, $\hat{A} (\hat z) +\hat{B}_V$ has exactly $q_2q_3\cdots q_d$ distinct eigenvalues, $\hat{\lambda}_V^l(\hat z)$, $l\in \hat W$, satisfying
	\begin{equation}\label{g28new}
	\hat{\lambda}_{V}^l(\hat z)=\left( \sum_{j=2}^d \rho^j_{l_j}z_j\right)+O(1).
	\end{equation}
	Let $ \hat {P}(\hat z,\lambda)=\det (\hat A(\hat z)+\hat B_V-\lambda I)$. Thus,
	\begin{equation}\label{g24}
	\hat {P}(\hat z,\lambda)=\prod_{l\in \hat W}(\hat{\lambda}_{V}^l(\hat z)-\lambda).
	\end{equation}
	Due to the simplicity of the eigenvalues, $\hat{\lambda}_V^l(\hat z)$, and \eqref{g24}, it is clear that
	$\partial_{\lambda} \hat {P}(\hat z,\lambda) \neq 0$ for $\lambda=\hat{\lambda}_V^l(\hat z)$, $l
	\in \hat W$, and $\hat z\in \hat{\Omega}$.
	By the inverse function theorem, we conclude that $\hat{\lambda}_{V}^l(\hat z)$ is holomorphic in $\hat{\Omega}$.
	
\end{proof}
\begin{lemma}\label{keylem}
	For $z\in\Omega$, 
	the matrix $A(z)+B_V$ has distinct eigenvalues $\lambda^l(z)$, where $l=(l_1,l_2,\cdots,l_d)\in W$ satisfying that $\lambda^l(z)$ is holomorphic in $\Omega$. Moreover, up to a relabeling, the eigenvalues $\lambda_V^l(z)$, $l\in W$, have the following representations (Laurent series expansions in variable $z_1$),
\begin{align}
\lambda_{V}^l(z)&=\left(\sum_{j=1}^d \rho^j_{l_j}z_j \right)+O(1)\label{g25}\\
&= \rho^1_{l_1}z_1 +\hat{\lambda}_V^l (\hat z)+\sum_{m=1}^{\infty}\frac{c_m(\hat z)}{z_1^m} \label{g8}, 
\end{align}
where the coefficient $c_m(\hat z)$ (depending on $l$), $m=1,2,\cdots$,  is holomorphic in $\hat{\Omega}$.

\end{lemma}

\begin{proof}
 
Fixing any $l\in W$, for $d\geq 2$,  let $\tilde{\lambda}^l(z)$ be 
	\begin{equation}\label{g120}
	 \tilde{\lambda}^l(z)= \rho^1_{l_1} z_1  +\left(\sum_{j=2}^d \rho^j_{l_j}z_j \right),
	\end{equation}
and for $d=1$, let 
\begin{equation}\label{g120new}
\tilde{\lambda}^l(z)=  \rho^1_{l_1} z_1 .
\end{equation}

Recall that (Lemma \ref{lesep}), 
\begin{equation}\label{g16}
A(z)+B_V=\text{ diag }( \tilde{\lambda}^l(z)) +O(1).
\end{equation}

By  Lemma \ref{leper}, \eqref{g52} and \eqref{g16},  for any  $z\in\Omega$,
 $A(z)+B_V$  has exactly $Q$ distinct eigenvalues  $\lambda_V^l(z)$,  $ l=(l_1,l_2,\cdots,l_d)\in W$  and
\begin{equation}\label{g17}
\lambda_{V}^l(z) =  \left(\sum_{j=1}^d \rho^j_{l_j}z_j \right)+O(1).
\end{equation}
We finish the proof of \eqref{g25}. 
We are going to prove  that $ \lambda_{V}^l(z) $ is holomorphic in $\Omega$ and \eqref{g8} holds.
Without loss of generality assume $l=(0,0,\cdots,0)\in W$.
Let $\tilde{\Omega}=\{z\in\C  : |z| \leq \frac{1}{C_1}\}$ for $d=1$ and  $\tilde{\Omega}=\{z\in\C ^d: |z_1| \leq \frac{1}{C_1}, \hat{z}\in \hat{\Omega}\}$ for $d\geq 2$. 

We  only prove the case $d\geq 2$. The case $d=1$ follows a similar, albeit simpler, argument.
Let $\lambda=\lambda_1+z_1$ and $\tilde{\mathcal P}^1_V(z,\lambda_1) =\tilde{\mathcal P}_V(z,\lambda_1+z_1)$. By  Lemma \ref{lesep}, direct computations imply that 
\begin{align}
\tilde{\mathcal P}^1_V(z,\lambda_1) =&\det (\hat A(\hat z)+\hat{B}_V-\lambda_1)\left(z_1^{Q-q_2q_3\cdots q_d}\prod_{l_1=1}^{q_1-1} (\rho^1_{l_1} -1) ^{q_2q_3\cdots q_d}\right)\label{g31}\\
&+\text {lower order terms of } z_1.\label{g32}
\end{align}
Let $G(z,\lambda_1)=z_1^{Q-q_2q_3\cdots q_d}\tilde{\mathcal P}^1_V(z_1^{-1},\hat z, {\lambda}_1)$. 
By \eqref{g31} and \eqref{g32}, $G(z,\lambda_1)$ is a polynomial  of $z_1$ and the constant term (with respect to $z_1$) is 

\begin{align}
G(0,\hat z, \lambda_1)&=\det (\hat A(\hat z)+\hat{B}_V-\lambda_1)\left(\prod_{l_1=1}^{q_1-1} (\rho^1_{l_1} -1) ^{q_2q_3\cdots q_d}\right)\\
&=\prod_{l\in \hat W}(\hat{\lambda}_{V}^l(\hat z)-\lambda_1)\left(\prod_{l_1=1}^{q_1-1} (\rho^1_{l_1} -1) ^{q_2q_3\cdots q_d}\right).\label{g41}
\end{align}
We are going to verify that (Claim 1) if for some $(z,\lambda_1) \in \tilde{\Omega}\times \C$, $G(z,\lambda_1)=0$, then
$\partial_{\lambda_1} G(z,\lambda_1)\neq 0$.  

When $z_1=0$,  Claim 1 is  true by \eqref{g41} and the simplicity of eigenvalues $\hat {\lambda}_V^l$, $l\in \hat W$.

When $z_1\neq 0$, 
\begin{align}
G(z,\lambda_1)&=z_1^{Q-q_2q_3\cdots q_d}\tilde{\mathcal P}_V(z_1^{-1},\hat z, z_1+{\lambda}_1)\nonumber\\
&= z_1^{Q-q_2q_3\cdots q_d} \prod_{l\in W} (\lambda^l_V(z_1^{-1}, \hat z) -z_1-\lambda_1).\label{g61}
\end{align}
In this case,  Claim 1 follows from \eqref{g61} and the simplicity of eigenvalues $\lambda_V^l$, $l\in W$.

Solve $G(z,\lambda_1)=0$ with the initial data $z_1=0$, any fixed  $\hat z\in \hat{\Omega}$ and $ \lambda_1= \hat \lambda_V^l(\hat z)$  
with $l=(0,0,\cdots,0)$.
By Claim 1 and inverse function theorem,  there exists a holomorphic solution $\lambda_1(z)$, $z\in \tilde{\Omega}$
such that $\lambda_1(0,\hat z)=\hat \lambda_V^l(\hat z)$.
It is easy to see that for any $z\in \Omega$, 
\begin{equation}
\lambda^{(0,0,\cdots,0)}_V(z)=z_1+ \lambda_1(z_1^{-1},\hat z).
\end{equation}
We finish the proof.

\end{proof}
\begin{remark}
	Since $\tilde{\mathcal P}_V(z,\lambda)=\det ( A(z)+ B_V-\lambda I)$  is a Laurent polynomial of $z_1^{q_1}$,
	$z_2^{q_2},\cdots,z_d^{q_d}$, we have that if $\lambda_V^{l} (z), z\in \Omega$ is an eigenvalue of  $A(z)+B_V$,   then for any $n\in W$, $\lambda_V^{l} (\rho^1_{n_1} z_1, \rho^2_{n_2} z_2,\cdots, \rho^d_{n_d} z_d), z\in \Omega$ is also an eigenvalue of $ A(z)+B_V$. By \eqref{g25}, we conclude that  for any $l\in W\backslash\{(0,0,\cdots,0)\}$,
	\begin{equation*}
\lambda_{V}^l(z)=\lambda_V^{(0,0,\cdots,0)} (\rho^1_{l_1} z_2, \rho^2_{l_2} z_1,\cdots, \rho^d_{l_d} z_d).
	\end{equation*}

\end{remark}
\section{Proof of Theorems \ref{mainthm2} and  \ref{mainthm1}}\label{S4}
 Clearly, Part 1 of Theorem \ref{mainthm2} immediately implies Part 2. So, 
in order to prove Theorem \ref{mainthm2}, after changing the variables $z_j=e^{2\pi i k_j}$ and $z_j\to z_j^{q_j}$,  $j=1,2,\cdots,d$, we only need to prove
\begin{theorem}\label{mainthm3}
If $V$ is a complex-valued $\Gamma$-periodic function and there exists a holomorphic function $\tilde f(z)$ in $(\C^{\star})^d$ such that $\lambda=\tilde{f}(z)$ is an eigenvalue of $\tilde{\mathcal{D}}_V(z)$ for any $z\in (\C^{\star})^d$, then
there exist  $l\in W$ and a constant $K$ such that 
	\begin{equation}\label{g42}
	\tilde{\mathcal{P}}_V(z, \lambda) = \prod_{n\in W}\left( K -\lambda+\sum_{j=1}^d \left(e^{2\pi \frac{n_j}{q_j} i}z_j +e^{-2\pi \frac{n_j+l_j}{q_j} i} z_j^{-1}\right) \right).
	\end{equation}
	
\end{theorem}
\begin{proof}[\bf Proof of Theorem \ref{mainthm3}]
 Assume that there exists a holomorphic function    $\tilde f(z)$  in $(\C^{\star})^d$ such that  for any $z\in (\C^\star)^d$, $\tilde{\mathcal{P}}_V(z, \tilde f(z))= 0$. 
 Laurent series expansions of $\tilde f(z)$  in variable $z_1$ lead to 
\begin{equation}\label{g81}
\tilde f(z)=  \sum_{m\in\Z} \tilde f_m(\hat z) z_1^m,
\end{equation}
where the coefficient $\tilde f_m(\hat z)$, $m\in\Z$  is holomorphic in $ (\C^{\star})^{d-1}$. 
In $\Omega$, by  Lemma \ref{keylem},  there exists $l\in W$ such that 
\begin{equation}\label{g9}
\tilde f(z)= \lambda_V^l (z).
\end{equation}

By Lemma \ref{keylem},  one has that $\tilde f_m(\hat z)=0$ in $\hat z\in \hat{\Omega}$ for any $ m\geq 2$ and $\tilde f_1(\hat z)=\rho^1_{l_1} $.  This implies that 
$\tilde f_m(\hat z)=0$ in $\hat z\in (\C^{\star}) ^{d-1}$ for any $ m\geq 2$ and $\tilde f_1(\hat z)=\rho^1_{l_1} $.
Changing the variable $z_1$ to $z_1^{-1}$ and repeating the proof, one has that 
$\tilde f_m(\hat z)=0$ in $\hat z\in (\C^{\star}) ^{d-1}$ for any $ m\leq -2$ and 
there exits   $l_1'\in\{0,1,\cdots, q_1-1\}$  such that $\tilde f_{-1}(\hat z)=\rho^1_{-l_1'}$.

We conclude that there exist  $l_1\in \{0,1,\cdots, q_1-1\}$ and $l_1'\in\{0,1,\cdots, q_1-1\}$ such that
\begin{equation}\label{g10}
\tilde f(z)=  \rho^1_{-l_1'}z_1^{-1}+ \rho^1_{l_1} z_1+ f_0(\hat z).
\end{equation}

Interchanging  $z_j$ and $z_1$, $j=2,3,\cdots, d$ and following the proof of \eqref{g10},
one has that  there exist $l\in W$ and $l'\in W$ such that 
\begin{equation}\label{g11}
\tilde f(z)=    K +\sum_{j=1}^d \left(\rho^j_{l_j}z_j +\rho^j_{-l'_j}z_j^{-1}\right),
\end{equation}
where $K$ is a constant.

Since $\tilde{\mathcal{P}}_V(z,\lambda)=\mathcal{P}_V(z_1^{q_1},z_2^{q_2},\cdots,z_d^{q_d},\lambda)$, we have that
for any $n\in W$, $$\tilde f(\rho_{n_1}^1 z_1,\rho_{n_2}^2 z_2, \cdots, \rho_{n_d}^d z_d)$$ is also an eigenvalue of $ \tilde{\mathcal D}_V(z)$. 
Therefore, \eqref{g11} implies \eqref{g42}.

	\end{proof}
 \begin{proof}[\bf Proof of Theorem \ref{mainthm1}]
 By \eqref{g451}, Part 1 immediately implies Part 2. So in order to prove Theorem  \ref{mainthm1}, 
it suffices to show that Part 2 implies Part 1.  
Since $V$ is real,  by Theorem \ref{mainthm2}, one has that the eigenvalues   
$$K +\sum_{j=1}^d \left(e^{2\pi \frac{n_j+k_j}{q_j} i} +e^{-2\pi \frac{n_j+l_j+k_j}{q_j} i}\right)$$
 are real  for all $k=(k_1,k_2,\cdots,k_d)\in\R^d$.  Therefore, we must have  that $K$ is real and  $l=(0,0,\cdots,0)$  in \eqref{g43}, namely
	\begin{equation}\label{g44}
\det (D_V(k)-\lambda I ) = \prod_{n\in W}\left( K -\lambda+\sum_{j=1}^d \left(e^{2\pi \frac{n_j+k_j}{q_j} i} +e^{-2\pi \frac{n_j+k_j}{q_j} i}\right) \right).
\end{equation}


By \eqref{g44} and \eqref{g451}, $V$ and the constant function ${\bf K}$ are Floquet isospectral (e.g. \cite{liufloquet23}). An Ambarzumian type theorem (e.g. \cite[Theorem 2]{kap3} or \cite[Section 3]{hat})  concludes that $V$ is the constant function ${\bf K}$.

 \end{proof}
\section{Proof of Theorem \ref{mainthm4}}\label{S5}
Basic facts of linear algebra lead to the following lemma:
 \begin{lemma}\label{keylem3}
Two $\Gamma$-periodic functions $V$ and $Y$ are Floquet isospectral if and only if  for all  $k\in\C^d$ and $\lambda\in\C$, 
\begin{equation}\label{gnew6}
	\det (D_V(k)-\lambda I ) =	\det (D_Y(k)-\lambda I ). 
\end{equation}
 \end{lemma}
 \begin{lemma}\label{keylem4}
If a $\Gamma$-periodic function $V$ with zero mean satisfies one of the two statements in Theorem \ref{mainthm2}, then
the constant $K$ in \eqref{g43} must be $0$.

\end{lemma}

\begin{proof}
Comparing the coefficients of $\lambda^{Q-1}$, we have
\begin{equation*}
K=\frac{1}{Q} \sum_{n\in W} V(n)=0.
\end{equation*}
\end{proof}
 \begin{lemma}\label{keylem1}
 Let $d=1$. Fix $l_1\in \{0,1,\cdots, q_1-1\}$. Then the following two statements are equivalent:
 	\begin{enumerate}
 		\item  A $\Gamma$-periodic function $V$ satisfies for all  $\lambda\in\C$, $$	\det (D_V(0)-\lambda I ) =\prod_{m=0}^{q_1-1}\left( e^{2\pi \frac{m}{q_1} i}+e^{-2\pi \frac{m+l_1}{q_1} i}-\lambda \right).$$
 		\item   A $\Gamma$-periodic function $V$ satisfies  for all  $\lambda\in\C$ and $k\in\C$, $$	\det (D_V(k)-\lambda I ) =\prod_{m=0}^{q_1-1}\left( e^{2\pi \frac{m+k}{q_1} i}+e^{-2\pi \frac{m+l_1+k}{q_1} i}-\lambda \right).$$
 	\end{enumerate}
 
\end{lemma}
\begin{proof}
Recall that $z=e^{\frac{2\pi ik}{q_1}}$ and $\tilde{\mathcal P}(z,\lambda)=	\det (D_V(k)-\lambda I ) $.
Direct computations  show that 
\begin{align}
\tilde{\mathcal P}(z,\lambda)=& (-1)^{q_1+1}e^{2\pi ik}+(-1)^{q_1+1}e^{-2\pi ik}+h_V(\lambda)\nonumber\\
	  =& (-1)^{q_1+1}z^{q_1}+(-1)^{q_1+1} z^{-q_1}+h_V(\lambda),\label{gnew2}
\end{align}
where $h_V(\lambda)$ is a polynomial of $\lambda$ with coefficients depending on $V$ (not depending on $k$).
Let $$T(z,\lambda)=\prod_{m=0}^{q_1-1}( ze^{2\pi \frac{m}{q_1} i}+z^{-1}e^{-2\pi \frac{m+l_1}{q_1} i}-\lambda ).$$
It is clear to see that $T(z,\lambda)$ is a Laurent polynomial in variable $z$ with the highest degree term $(-1)^{q_1+1} z^{q_1}$ and 
the lowest degree term $(-1)^{q_1+1} z^{-q_1}$. 
Noting that $T(z,\lambda)=T(e^{2\pi \frac{m}{q_1} i}z,\lambda )$ for all $m\in\Z$, we conclude that $T(z,\lambda)$ is a Laurent polynomial in variable $z^{q_1}$. Therefore, there exists $T_0(\lambda)$, a polynomial in $\lambda$, such that 
\begin{equation}\label{gnew3}
T(z,\lambda)=(-1)^{q_1+1}z^{q_1}+(-1)^{q_1+1} z^{-q_1}+T_0(\lambda).
\end{equation} 
By \eqref{gnew2} and \eqref{gnew3},
 $\tilde{\mathcal P}(z,\lambda)= T(z,\lambda)$  for all $z\in\C^\star$ and $\lambda \in \C$ if and only if 
$h_V(\lambda)= T_0(\lambda) $  for all  $\lambda \in \C$. This completes the proof.

\end{proof}
The following lemmas   have been  proved by  an algebraic approach in ~\cite{fri}. 
\begin{lemma}\label{keythm}\cite{fri} 
Given an $N \times N$ complex matrix $M$ and an arbitrary set of $N$ complex numbers $\eta_m$, $m=0,1,2,\cdots,N-1$. There exists    a  $N\times N$ diagonal matrix   $\tilde M$ such that the eigenvalues of $M + \tilde M$ are precisely 
 $\eta_m$, $m=0,1,2,\cdots,N-1$.
\end{lemma}
\begin{lemma}\label{keylem5}\cite{fri} 
	Fix any $\Gamma$-periodic function $V$ and $k_0\in\C^d$. 
	Then there are at most $Q!$ functions $Y$ satisfying $	\det (D_Y(k_0)-\lambda I ) =	\det (D_V(k_0)-\lambda I )$ for all $\lambda\in\C$. In particular, for any fixed $\Gamma$-periodic function $V$, 
	\begin{equation*}
\#\{ Y\in X: Y \sim V\}\leq Q!.
	\end{equation*}
\end{lemma}
\begin{remark}
  The proof of both Lemmas \ref{keythm} and \ref{keylem5} can be reduced to studying the solutions of a system of polynomial equations with the following form:
  \begin{equation}\label{g200}
  \sigma_i(x_1,x_2,\cdots, x_N)+g_i(x_1,x_2,\cdots, x_N)=0, \quad i=1,2,\cdots,N,
  \end{equation}
  where $\sigma_i$, $i=1,2,\cdots,N$, represents the $N$ elementary symmetric polynomials of $x_1,x_2,\cdots, x_N$, and the degree of each $g_i$, $i=1,2,\cdots,N$, is less than $i$. Friedland initially proved Lemmas \ref{keythm} and \ref{keylem5} using an algebraic approach~\cite{fri}. Subsequently, Lemmas \ref{keythm} and \ref{keylem5} were proved using different approaches: topological degree arguments ~\cite{ale} and algebraic geometry ~\cite{fri2}. Additionally, Lemma \ref{keylem5} was established by Kappeler in ~\cite{kap1,kap3}.

\end{remark}
\begin{proof}[\bf Proof of Theorem \ref{mainthm4}]

Since the cardinality of $W$ is $Q$, 	by  Lemmas \ref{keylem3} and \ref{keylem4}, one has that 
$\# \{ X_e/\sim\}\leq Q$.  	By Lemma \ref{keylem5},   \eqref{gnew58}  holds. 
Therefore, in order to prove Theorem \ref{mainthm4}, by Lemma \ref{keylem4}, we only need to show that 
for any $ l\in W$, there exists a complex-valued $\Gamma$-periodic function $V$ such that 
\begin{equation}\label{gnew43}
\det (D_V(k)-\lambda I ) = \prod_{n\in W}\left( -\lambda+\sum_{j=1}^d \left(e^{2\pi \frac{n_j+k_j}{q_j} i} +e^{-2\pi \frac{n_j+l_j+k_j}{q_j} i}\right) \right).
\end{equation}
By constructing separable functions with the form $V(n)=\sum_{j=1}^dV_j(n_j)$, it suffices to prove the existence of functions $V$ such that \eqref{gnew43} holds for   $d=1$. 
Assume $d=1$ ($\Gamma=q_1\Z$).   Fix $l_1\in \{0,1,\cdots, q_1-1\}$.
Applying Lemma \ref{keythm} with $N=q_1$, $M=D_{\bf 0} (0)$, $\eta_m= e^{2\pi \frac{m}{q_1} i}+e^{-2\pi \frac{m+l_1}{q_1} i}$, $m=0,1,2,\cdots, q_1-1$, there exists a $\Gamma$-periodic function  $V$ such that
$D_V(0)$ has eigenvalues $\eta_m= e^{2\pi \frac{m}{q_1} i}+e^{-2\pi \frac{m+l_1}{q_1} i}$, $m=0,1,2,\cdots, q_1-1$. This implies that  for all $\lambda\in\C$, 
\begin{equation}\label{gnew4}
\det (D_V(0)-\lambda I ) =\prod_{m=0}^{q_1-1}\left( e^{2\pi \frac{m}{q_1} i}+e^{-2\pi \frac{m+l_1}{q_1} i}-\lambda \right).
\end{equation}
By Lemma \ref{keylem1} and \eqref{gnew4}, one has that  for all $\lambda\in\C$ and $k\in\C$, 
\begin{equation}\label{gnew5}
\det (D_V(k)-\lambda I ) =\prod_{m=0}^{q_1-1}\left( e^{2\pi \frac{m+k}{q_1} i}+e^{-2\pi \frac{m+l_1+k}{q_1} i}-\lambda \right).
\end{equation}
We finish the proof.
\end{proof}

	\section*{Acknowledgments}
W. Liu was a 2024-2025 Simons fellow.
The  research was supported in part by NSF DMS-2000345, DMS-2052572 and DMS-2246031.

\section*{Statements and Declarations}
{\bf Conflict of Interest} 
The author  declares no conflicts of interest.

\vspace{0.2in}
{\bf Data Availability}
Data sharing is not applicable to this article as no new data were created or analyzed in this study.

\end{document}